\documentclass[12pt]{amsart}
\usepackage{amssymb,mathrsfs,longtable}
\usepackage[matrix,arrow,curve]{xy}
\usepackage{setspace}
\usepackage{multirow}
\usepackage{textcomp}
\usepackage{hyperref}

\usepackage[dvipsnames]{xcolor}
\usepackage{graphicx}

\usepackage[normalem]{ulem}
\usepackage{cancel}

\newcounter{cequation}[section]

\newtheorem{theorem}[cequation]{Theorem}
\newtheorem*{theorem*}{Theorem}
\newtheorem{lemma}[cequation]{Lemma}
\newtheorem{corollary}[cequation]{Corollary}

\newtheorem{proposition}[cequation]{Proposition}

\theoremstyle{definition}
\newtheorem{example}[cequation]{Example}

\newtheorem*{definition*}{Definition}
\newtheorem{question}[cequation]{Question}

\newtheorem*{notation*}{Notation}

\theoremstyle{remark}
\newtheorem{remark}[cequation]{Remark}

\makeatletter\@addtoreset{equation}{section}
\makeatletter\@addtoreset{section}{part}

\makeatother

\def \O {\mathcal{O}}
\def \CC {\mathbb{C}}
\def \P {\mathbb{P}}
\def \PP {\mathbb{P}}
\def \AA {\mathbb{A}}
\def \QQ {\mathbb{Q}}
\def \ZZ {\mathbb{Z}}
\def \Cl {\mathrm{Cl}}
\def \Pic {\mathrm{Pic}}

\def \Aut {\mathrm{Aut}}
\def \GL {\mathrm{GL}}

\def \ge {\geqslant}
\def \le {\leqslant}

\title{On automorphisms of quasi-smooth weighted complete intersections}

\author{Victor Przyjalkowski and Constantin Shramov}

\address{\emph{Victor Przyjalkowski}
\newline
\textnormal{Steklov Mathematical Institute of RAS, 8 Gubkina street, Moscow 119991, Russia.
}
\newline
\textnormal{National Research University Higher School of Economics, Laboratory of Mirror Symmetry, NRU HSE, 6 Usacheva street, Moscow, 117312, Russia.
}
\newline
\textnormal{\texttt{victorprz@mi.ras.ru, victorprz@gmail.com}}}

\address{\emph{Constantin Shramov}
\newline
\textnormal{Steklov Mathematical Institute of RAS,
8 Gubkina street, Moscow 119991, Russia.
}
\newline
\textnormal{National Research University Higher School of Economics, Laboratory of Algebraic Geometry, NRU HSE, 6 Usacheva str., Moscow, 117312, Russia.
}
\newline
\textnormal{\texttt{costya.shramov@gmail.com}}}

\thanks{Victor Przyjalkowski was supported by
Laboratory of Mirror Symmetry NRU HSE, RF government
grant, ag. \textnumero 14.641.31.0001.
Constantin Shramov was supported by the
HSE University Basic Research Program,
Russian Academic Excellence Project~\mbox{``5-100''}.
Both authors are supported by the Foundation for the
Advancement of Theoretical Physics and Mathematics ``BASIS''.}

\begin{document}

\begin{abstract}
We show that every reductive subgroup of the automorphism group of a quasi-smooth well formed weighted complete intersection
of dimension at least $3$
is a restriction of a subgroup in the automorphism group in the ambient weighted projective space.
Also, we provide examples demonstrating that
the automorphism group of a  quasi-smooth
well formed Fano weighted complete intersection may be infinite
and even non-reductive.
\end{abstract}

\maketitle

\section{Introduction}

One of the ways to obtain interesting examples of Fano varieties is
to construct them as complete intersections in weighted projective spaces.
We refer the reader to~\cite{Do82}
and~\cite{IF00} (or to~\S\ref{section:preliminaries} below) for
definitions and basic properties of weighted projective spaces
and complete intersections therein.
One of the advantages of such constructions
is that many properties of the resulting varieties are easy to analyse.
In particular, the following is known about automorphism
groups of weighted complete intersections.

\begin{theorem}[{\cite[Theorem~1.3]{PrzyalkowskiShramov-AutWCI}}]
\label{theorem:automorphisms}
Let $X$ be a smooth well formed
weighted complete intersection of dimension~$n$.
Suppose that either $n\ge 3$, or $K_X\neq 0$.
Then the group~\mbox{$\Aut(X)$} is finite unless
$X$ is isomorphic either to $\P^n$ or to a quadric hypersurface in~$\P^{n+1}$.
\end{theorem}

As a by-product of Theorem~\ref{theorem:automorphisms}, one obtains the following.

\begin{corollary}
\label{corollary:automorphisms-smooth-reductive}
Let $X$ be a smooth well formed
weighted complete intersection.
Suppose that either $\dim X\ge 3$, or $K_X\neq 0$.
Then the group~\mbox{$\Aut(X)$} is reductive.
\end{corollary}

Although smooth varieties are most natural to study,
in many situations it makes sense to consider weighted complete
intersections with a slightly weaker property, namely, quasi-smooth ones
(see~\S\ref{section:preliminaries}).
The main purpose of this paper is to prove the following result
on automorphism groups of quasi-smooth weighted complete intersections.

\begin{theorem}\label{theorem:induced-action}
Let $\PP$ be a well formed weighted projective space, and let $X\subset\PP$ be a
quasi-smooth well formed weighted complete intersection
which is not an intersection with a linear cone.
Suppose that either $\dim X\ge 3$, or $\dim X=2$ and $K_X\neq 0$,
or $X$ is a rational curve. Then $\Aut(X)$ is a linear algebraic group.
Furthermore, let $\Gamma$ be a reductive linear algebraic subgroup in $\Aut(X)$.
Then there is an action of $\Gamma$ on $\PP$ that restricts to the action of $\Gamma$ on $X$.
\end{theorem}

We point out that the assertion of Theorem~\ref{theorem:induced-action}
fails for curves of genus~$1$ that are complete intersections,
see Example~\ref{example:elliptic-fail}.

As a straightforward application of Corollary~\ref{corollary:automorphisms-smooth-reductive} and Theorem~\ref{theorem:induced-action},
we obtain the following assertion.

\begin{corollary}\label{corollary:smooth-induced}
Let $X\subset\PP$ be a smooth well formed
weighted complete intersection which is not an intersection with a linear cone.
Suppose that either $\dim X\ge 3$, or~\mbox{$\dim X=2$} and~\mbox{$K_X\neq 0$},
or $X$ is a rational curve.
Then there is
an action of the group $\Aut(X)$ on~$\PP$ that restricts to the action of $\Aut(X)$ on $X$.
\end{corollary}

It would be interesting to know the answer to the following question (cf. Example~\ref{example:non-reductive}).

\begin{question}
Does the second assertion of
Theorem~\ref{theorem:induced-action} hold for the whole automorphism group
$\Gamma=\Aut(X)$
without the assumption that $\Gamma$ is reductive?
\end{question}

\smallskip

The plan of the paper is as follows.
In~\S\ref{section:preliminaries} we collect some auxiliary facts
about weighted complete intersections.
In~\S\ref{section:automorphisms}
we collect several auxiliary results on automorphism groups of weighted
complete intersections.
In~\S\ref{section:automorphisms-restricted}
we prove Theorem~\ref{theorem:induced-action}.
In~\S\ref{section:infinite-examples} we provide examples demonstrating that
an automorphism group of a  quasi-smooth
well formed Fano weighted complete intersection may be infinite
and even non-reductive (so that the assertion of
Corollary~\ref{corollary:automorphisms-smooth-reductive} does not hold
in this case).

\smallskip
We are grateful to I.\,Cheltsov, R.\,Hartshorne, T.\,Okada, Yu.\,Prokhorov,
and D.\,Timashev for useful discussions. Special thanks go to the anonymous referee who spotted
several gaps in the first draft of the paper.

\section{Preliminaries}
\label{section:preliminaries}

In this section we collect some auxiliary facts
about weighted complete intersections.

Put
$$
\P=\P(a_0,\ldots,a_N)=\mathrm{Proj}\, \CC[x_0,\ldots,x_N],
$$
where the weight of $x_i$ equals $a_i$.
Without loss of generality we assume that $a_0\le\ldots\le a_N$.
We will use the abbreviation
\begin{equation*}
(a_0^{r_0},\ldots,a_M^{r_M})=
(\underbrace{a_0,\ldots,a_0}_{r_0\ \text{times}},\ldots,\underbrace{a_M,\ldots,a_M}_{r_M\ \text{times}}),
\end{equation*}
where $r_0,\ldots,r_M$ will be allowed to be
any positive integers. If some of $r_i$ is equal to $1$ we drop it for simplicity.

We say that a subvariety $X\subset\P$ of codimension $k\ge 1$ is a \emph{weighted complete
intersection of multidegree $(d_1,\ldots,d_k)$} if its weighted homogeneous ideal in $\CC[x_0,\ldots,x_N]$
is generated by a regular sequence of $k$ homogeneous elements of degrees $d_1,\ldots,d_k$.
The regularity of the above sequence is equivalent to the requirement that the codimension of (every irreducible component of) the variety $X$
equals~$k$, see, for instance,~\cite[\S2]{PSh20a}.

We put some natural restrictions on $\PP$ and $X$ to avoid too bad
complete intersections.
We say that $X$ is \emph{well formed} if the following two conditions hold. First, we require that~$\PP$ is well formed, that is,
the greatest common divisor of any $N$ of the weights~$a_i$ equals~$1$.
Second,
$$
\mathrm{codim}_X \left( X\cap\mathrm{Sing}\,\P \right)\ge 2.
$$
Note that the singular locus of $\PP$ is a union of some coordinate strata.

We say that~$X$ is \emph{an intersection
with a linear cone} if $d_j=a_i$ for some~$i$ and~$j$.

\begin{remark}\label{remark:coordinate}
Let $X\subset\PP=\PP(a_0^{r_0},\ldots,a_M^{r_M})$,
where $a_0<\ldots<a_M$,
be a weighted complete intersection which is not
an intersection with a linear cone.
Write $\PP\cong\operatorname{Proj}(R(\PP))$,
where
$$
R(\PP)=
\CC\left[x_{0,1},\ldots,x_{0,r_0},\ldots,x_{M,1},\ldots,x_{M,r_M}\right],
$$
so that $x_{i,p}$ is a 
coordinate of weight
$a_i$ on $\PP$.
Then none of the coordinates $x_{i,p}$ vanishes on $X$.
Indeed, if $x_{i,p}$ vanishes on $X$, then it is contained in the weighted homogeneous
ideal $I\subset R(\PP)$ of $X$. On the other hand,
$x_{i,p}$ is not contained in the ideal~\mbox{$I'\subset I$} generated
by the coordinates $x_{0,1},\ldots,x_{i-1,r_{i-1}}$ of smaller weight.
Thus one of the defining equations of $X$ must have degree equal to $a_i$,
which is not the case by assumption. Also, changing the coordinates, we deduce from this that if
a weighted homogeneous polynomial $f$ of weighted degree $a_i$ vanishes on $X$, then
$f$ depends only on the variables $x_{0,1},\ldots,x_{i-1,r_{i-1}}$ of weights smaller than~$a_i$.
\end{remark}

Every subvariety in a weighted projective space naturally comes together with a cone over it.
That is, let
$$
\AA=\mathrm{Spec}\,\CC[x_0,\ldots,x_N]\cong\AA^N.
$$
Then $\PP=\left(\AA\setminus \{0\}\right)/\CC^*$, where $\CC^*$ naturally acts
by the weights $a_0,\ldots,a_N$.
Denote the projection $\AA\setminus \{0\}\to \PP$ by $\pi$.
For a subvariety $X\subset\PP$, let $C_X$ be the closure in $\AA$ of the preimage $\pi^{-1}(X)\subset\AA\setminus\{0\}$
of $X$; we will call $C_X$ the \emph{affine cone over~$X$}.
One says that $X$ is \emph{quasi-smooth}, if the affine cone over it is smooth outside the origin.

\begin{lemma}[{cf. \cite[Theorem 3.4]{Ha62}}]
\label{lemma:WCI-connected}
Let $X\subset\PP$ be a positive-dimensional weighted complete intersection.
Then~$X$ is connected. Moreover, if $X$ is quasi-smooth,
then it is irreducible.
\end{lemma}

\begin{proof}
First let us prove the connectedness assertion. 
Suppose that $X$ is not connected.
Write $X=X'\cap X''$, where $X'$ and $X''$
are unions of irreducible components of $X$, and $X'\cap X''=\varnothing$.
Then $C_X=C_X'\cup C_X''$, where $C_X'$ and $C_X''$
are cones over $X'$ and $X''$, respectively. 
On the other hand, the cone $C_X$ is a complete intersection in the affine space $\AA$.
In particular, it is Cohen--Macaulay, see~\cite[\S18.5]{Eisenbud}.
Moreover, since it is a cone, it is obviously connected.
By Hartshorne's Connectedness Theorem (see~\cite[Theorem~18.12]{Eisenbud}),
the intersection $C_X'\cap C_X''$ has codimension $1$ in $C_X$. 
Since the dimension of any irreducible component of $C_X$ is $\dim (X)+1\ge 2$, this means
that $C_X'\cap C_X''$ contains some points outside the vertex of the cone $C_X$.
Thus $X'\cap X''\neq\varnothing$, which is a contradiction. 

Now assume that $X$ is quasi-smooth.
Suppose that it is
reducible. Let $X_1$ be one of its irreducible components.
Let $X_2$ be another one
that intersects $X_1$ at some point $P$,
which exists by connectedness of $X$.
Let $C_{X_1}, C_{X_2}\subset C_X$ be cones over these components.
Then the intersection of $C_{X_1}$ and $C_{X_2}$ contains
the affine cone over $P$ and, thus, $C_X$ is
singular along this cone. However, this contradicts the quasi-smoothness
assumption.
\end{proof}

\begin{remark}
An alternative proof of the first assertion of Lemma~\ref{lemma:WCI-connected}
can be obtained from the Lefschetz-type theorem,
see~\cite[Proposition~1.4]{Ma99}.
\end{remark}

Note that a general quasi-smooth weighted complete intersection of dimension at least~$3$ is isomorphic to a quasi-smooth well formed weighted complete
intersection
which is not an
intersection with a linear cone, see~\cite[Proposition 2.9]{PSh20b}.

Singularities of quasi-smooth well formed weighted complete intersections
can be easily described.

\begin{proposition}[{see \cite[Proposition 8]{Di86}}]
\label{proposition:singularities-of-X}
Let $X\subset \PP$ be a quasi-smooth well formed weighted complete intersection.
Then the singular locus of $X$ is
the intersection of $X$ with the singular locus of $\PP$.
\end{proposition}

\begin{remark}
Note that in Proposition~\ref{proposition:singularities-of-X} we can omit the assumption that $X$ is a weighted complete intersection.
The proof in this case is literally the same as the proof
of~\mbox{\cite[Proposition 8]{Di86}}.
\end{remark}

\begin{corollary}
\label{corollary:singularities-of-curves}
A quasi-smooth well formed weighted complete intersection of dimension~$1$
is smooth.
\end{corollary}

The converse assertion to Corollary~\ref{corollary:singularities-of-curves}
holds in arbitrary dimension.

\begin{lemma}[{\cite[Corollary 2.14]{PrzyalkowskiShramov-Weighted}}]
\label{lemma:smooth-vs-quasismooth}
Let $X\subset\PP$ be a smooth well formed weighted complete intersection.
Then $X$ is quasi-smooth.
\end{lemma}

If the weighted projective space $\PP$ is not well formed,
then the assertion of Lemma~\ref{lemma:smooth-vs-quasismooth}
may fail.

\begin{example}
The hypersurface $X$ in $\PP=\PP(1,2^n)$ with 
coordinates
$x_0,\ldots,x_n$ given by equation
$$
x_0^2x_1+x_2^2+\ldots+x_n^2=0
$$
is not quasi-smooth because the cone over it
in $\mathbb{A}^{n+1}$ is singular at the point~\mbox{$(0,1,0,\ldots,0)$}.
On the other hand, $X$ is
isomorphic to a quadric
in $\PP^n\cong \PP(1,2^n)$ with homogeneous coordinates $z_0,\ldots,z_n$
given by equation
$$
z_0z_1+z_2^2+\ldots+z_n^2=0,
$$
and thus it is smooth.
\end{example}

The reader may wonder if there exists a
smooth (but not well formed)
weighted complete intersection~$X$ in a well formed
weighted projective space~$\PP$ such that~$X$
is not quasi-smooth. The following example suggested to us by
I.\,Cheltsov and Yu.\,Prokhorov shows that the answer is positive.

\begin{example}
Let $X$ be a hypersurface in the well formed weighted projective
space~\mbox{$\PP=\PP(2,3,5^n)$} with coordinates
$x_0,\ldots,x_{n+1}$ given by equation $x_0^3=x_1^2$.
Obviously, it is not quasi-smooth (and not well formed).
We claim that $X$ is smooth. Indeed, it is enough to check
this in the neighborhood of the subset
defined by equations $x_0=x_1=0$. This subset
is covered by pairwise isomorphic affine charts given by $x_i=1$.
So consider the affine chart $U\subset\PP$
where $x_{n+1}=1$. This chart is a quotient of
the affine space $\AA^{n+1}$ with coordinates $u_0,\ldots,u_n$
by the group $\ZZ/5\ZZ$ whose generator multiplies the coordinates $u_0$ and
$u_1$ by $\varepsilon^2$ and $\varepsilon^3$, respectively,
where $\varepsilon$ is a non-trivial root of unity of degree~$5$,
and acts trivially  on the remaining coordinates.
The intersection~\mbox{$X\cap U$} is isomorphic to the quotient of
the subset of $\AA^{n+1}$ defined by equation~\mbox{$u_0^3=u_1^2$}.
The algebra of invariants of the above action on $\AA^{n+1}$
is generated by the functions
$$
u_0^5, u_1^5, u_0u_1, u_2,\ldots,u_n.
$$
Denoting them by $v_0,\ldots,v_{n+1}$, we see that
$U\cong\AA^{n+1}/(\ZZ/5\ZZ)$ is isomorphic to a hypersurface
given by equation $v_0v_1=v_2^5$ in the affine space $\AA^{n+2}$
with coordinates $v_0,\ldots,v_{n+1}$, and $X\cap U$ is isomorphic to
a subvariety of $\AA^{n+2}$ given by equations
$$
v_0-v_2^2=v_1-v_2^3=0.
$$
Obviously, these equations define a smooth variety.
\end{example}

Although the assertion of Corollary~\ref{corollary:singularities-of-curves} fails in dimensions
larger than~$1$, one can still show that singularities of quasi-smooth well formed weighted
complete intersections are relatively nice.

\begin{proposition}
\label{proposition:log-terminal}
Let~$X$ be a quasi-smooth
well formed weighted complete intersection.
Then $X$ is normal and has quotient singularities. In particular, the singularities of $X$ are log terminal.
\end{proposition}
\begin{proof}
Since $X$ is quasi-smooth, it has quotient singularities,
see for instance~\cite[\S6]{IF00}. Hence $X$ is normal.
Furthermore, quotient singularities are
log terminal by~\cite[Proposition~1.7]{Kaw84}.
\end{proof}

The divisorial sheaf $\O_\PP(1)$ is not necessary a line bundle.
The description of all line bundles on $\PP$ is given by the following
assertion, see \cite[Proposition 8]{RoTe12} or
the proof of~\mbox{\cite[Theorem 3.2.4(i)]{Do82}}.

\begin{proposition}
\label{proposition:Cartier}
Let~\mbox{$\PP=\PP(a_0,\ldots,a_N)$} be a well formed weighted
projective space.
Then the Picard group $\Pic(\PP)$ is a free group generated by $\O_\PP(l)$,
where $l$ is the least common multiple of the weights~$a_i$.
\end{proposition}

For a weighted complete intersection $X$ of multidegree $(d_1,\ldots,d_k)$
in~$\P(a_0,\ldots,a_N)$, define
\begin{equation*}
i_X=\sum a_j-\sum d_i.
\end{equation*}
Let $\omega_X$ be the dualizing sheaf on $X$.

\begin{theorem}[{see~\cite[Theorem 3.3.4]{Do82}, \cite[6.14]{IF00}}]
\label{theorem:adjunction}
Let $X$ be a quasi-smooth
well formed weighted complete intersection.
Then
$$
\omega_X\cong\O_X\left(-i_X\right).
$$
\end{theorem}

Theorem~\ref{theorem:adjunction} allows to understand the basic properties
of weighted complete intersections if the weights of the weighted projective spaces and the degrees of the defining equations are known. We illustrate
this by the following observation. Recall that a variety $X$ is called
\emph{rationally connected} if for two general points $P_1,P_2\in X$
there is a rational curve on $X$ passing through $P_1$ and $P_2$.
A variety is called \emph{uniruled} if it is covered by rational curves.

\begin{proposition}\label{proposition:rational-curves}
Let $X\subset\PP=\PP(a_0,\ldots,a_N)$ be a quasi-smooth well formed weighted
complete intersection. The following assertions hold.
\begin{itemize}
\item[(i)] If $i_X>0$, then $X$ is rationally connected.

\item[(ii)] If $i_X\le 0$ and $i_X$ is divisible by
all weights $a_i$,
then $X$ is not uniruled.

\item[(iii)] If $i_X=0$, then $X$ is not uniruled.
\end{itemize}
\end{proposition}

\begin{proof}
We know from Proposition~\ref{proposition:log-terminal} that
$X$ has log terminal singularities.

Suppose that $i_X>0$. Then the anticanonical
divisor $-K_X$ is ample by Theorem~\ref{theorem:adjunction}.
Therefore, assertion~(i) holds by~\cite{Zhang-RC}.

Now suppose that $i_X\le 0$ and $i_X$ is divisible by
all weights $a_i$.
Then the sheaf $\O_{\PP}(-i_X)$ is a line bundle by Proposition~\ref{proposition:Cartier}. Thus by
Theorem~\ref{theorem:adjunction}
the canonical class $K_X$ is an effective Cartier divisor.
In particular, the singularities of $X$ are Gorenstein, so that all discrepancies of $X$ are integers. Since the singularities of $X$ are also
log terminal, we conclude that the discrepancies are non-negative, which means that the singularities of $X$ are actually canonical.
Hence there exists a resolution of singularities $\widetilde{X}\to X$
such that the canonical divisor~$K_{\widetilde{X}}$ is effective.
This implies that $\widetilde{X}$ (and thus also $X$) is not uniruled,
see~\cite{MiyaokaMori}, and thus proves assertion~(ii).

Assertion~(iii) follows from assertion~(ii).
\end{proof}

We do not know if the divisibility assumption in
Proposition~\ref{proposition:rational-curves}(ii)
is inevitable. Note that without this assumption
Theorem~\ref{theorem:adjunction} does not imply that
the singularities of $X$ are Gorenstein. On the other hand, a variety with
non-Gorenstein log terminal singularities and
ample canonical class may be uniruled, as one can see from
the following example that was pointed out to us by Yu.\,Prokhorov and the anonymous referee.

\begin{example}
\label{example:Prokhorov}
Let $C\subset \PP^2$ be a nodal plane rational curve of degree $d$ (say, a general projection
of a rational normal curve of degree $d$). Then $C$ has $m=\frac{(d-1)(d-2)}{2}$ nodes.

Let $f\colon \widetilde{X}\to \PP^2$ be the blow up of the nodes, let $E_1,\ldots, E_m$ be the exceptional curves,
and let $E=\sum_{i=1}^m E_i$. Denote by $\widetilde{C}$ the proper  transform of $C$ on $\widetilde{X}$.
One has
$$
\widetilde{C}\sim f^*C-2E,
$$
so that
$$
\widetilde{C}^2= C^2+4E=-d^2+6d-4.
$$
Hence $\widetilde{C}^2<0$ for $d\ge 6$.
Now assume that $d> 6$ and denote $\widetilde{d}=\widetilde{C}^2$.
Since $\widetilde{C}$ is smooth and rational, there exists a contraction $g\colon \widetilde{X}\to X$
of $\widetilde{C}$ to a surface $X$ with a unique singular point.
The singular point of $X$ is locally isomorphic to the quotient of $\AA^2$ by a cyclic group of order $\widetilde{d}$ acting with
weights $(1,1)$; in particular, it is log terminal.

We have
$$
K_{\widetilde{X}} \sim f^* K_{\PP^2}+E\sim_{\QQ} -\frac{3}{d}f^*C+E\sim_{\QQ} -\frac{3}{d}\widetilde{C}+ \left(1-\frac{6}{d}\right)E.
$$
Thus after the contraction we get
$$
K_X\sim g_*K_{\widetilde X}\sim_{\QQ} \left(1-\frac{6}{d}\right)\sum_{i=1}^m g_*E_i.
$$
Since $d>6$, we conclude that $K_X$ is effective.

Denote $L_i=g_*E_i$ and $L=\sum_{i=1}^m L_i$,
so that
$$
K_X\sim_{\QQ} \left(1-\frac{6}{d}\right)L.
$$
Write
$E_i=g^*L_i-\alpha \widetilde{C}$.
Then
$$
2=E_i\cdot \widetilde{C}=-\alpha \widetilde{C}^2=-\alpha \widetilde{d}.
$$
By the projection formula one has
$$
L_i\cdot L=g^*L_i\cdot E= \left(E_i+\alpha \widetilde{C}\right)\cdot E=
-1+2\alpha m=\frac{d^2}{d^2-6d+4}>0.
$$
In particular, we see that $L^2>0$. Moreover, we claim that any irreducible effective curve $F$ on $X$ has positive intersection with $L$.
Indeed, this holds for $F=L_i$ by the inequality above.
Assume that $F\neq L_i$ for all $i$.
If $F$ intersects $L_i$ for some $i$, then its intersection with $L$ is positive.
Thus, we may assume that $F$ does not intersect $L_i$ for any $i$. In particular, it does not pass through the point $g(\widetilde{C})$. This means that
the proper transform $\widetilde F$ of $F$ on $\widetilde X$ does not intersect $\widetilde C$.
Also, it does not intersect $E_i$ for any $i$. Therefore, $f(\widetilde{F})$ does not intersect
$C$, which is absurd.
This, by Nakai--Moishezon criterion, shows that the class of $L$, and hence also $K_X$, is ample.
Finally, note that $X$ is rational by construction, so that in particular it is uniruled.
\end{example}

It appears that the divisor class group of a quasi-smooth well formed
weighted complete intersection has nice properties.

\begin{theorem}[{cf. \cite[Remark~4.2]{Okada2}, \cite[Proposition~2.3]{PST17}}]
\label{theorem:Okada}
Let~$X$ be a quasi-smooth
well formed weighted complete intersection.
Suppose that either $\dim X\ge 2$, or $X$ is a rational curve.
Then the divisor class group $\Cl(X)$ has no torsion.
Moreover, if $\dim(X)\ge 3$, then the group $\Cl(X)\cong \ZZ$ is generated by the class of~\mbox{$\mathcal{O}_X(1)$}.
\end{theorem}

We reproduce the proof of~\cite[Remark 4.2]{Okada2} with some modification suggested to us by T.\,Okada.

\begin{proof}[Proof of Theorem~\ref{theorem:Okada}]
If $X$ is a rational curve, then it is smooth by Corollary~\ref{corollary:singularities-of-curves}, so the assertion is obvious.
Suppose that $\dim X\ge 2$.
Let $C_X\subset\AA$ be the affine cone over~$X$, and let
$R$ be the (graded) coordinate algebra of $C_X$.
We have the exact sequence
\[\xymatrix{
0 \ar[r] & \ZZ \ar[r]^(.35){\theta} & \Cl(X)\ar[r] & \Cl(R) \ar[r] & 0,
}\]
where $\theta$ sends $1$ to a divisor class $\mathcal{O}_{X}(1)$, see e.\,g.~\cite[Theorem~1.6]{Wa81}.

Let $\mathfrak{m}$ be the maximal ideal of the origin of $C_X$.
Then by~\cite[Corollary~10.3]{Fo73}, one has $\Cl(R)\cong \Cl(R_{\mathfrak{m}})$.
If $\dim (X)\ge 3$, then $R_\mathfrak{m}$ is a complete intersection local ring of dimension at least $4$
which is regular outside the maximal ideal, so that $\Cl(R_{\mathfrak{m}})=0$, see~\cite[\S18]{Fo73}.
Thus $\Cl(X)\cong \ZZ$ is generated by the class of $\mathcal{O}_X(1)$.

Now assume that $\dim (X)=2$.
Set
$$
U = \mathrm{Spec}(R_{\mathfrak{m}}) \setminus \{\mathfrak{m}\}.
$$
By~\mbox{\cite[Proposition~18.10(b)]{Fo73}}, we have an isomorphism
$\Pic (U) \cong \Cl (R_{\mathfrak{m}})$. Finally, by
assertion (ii) of the main theorem of~\cite{Ro76},
the group $\Pic(U)$ is torsion free. It follows that
$$
\Cl(R) \cong \Cl(R_{\mathfrak{m}}) \cong\Pic(U)
$$
has no torsion, and hence the same holds for $\Cl(X)$.
\end{proof}

\begin{remark}\label{remark:curve-torsion}
The assertion of Theorem~\ref{theorem:Okada} obviously fails in the case when $X$ is a curve of positive genus.
\end{remark}

\section{Automorphisms}
\label{section:automorphisms}

In this section we collect several auxiliary results on automorphism groups of weighted
complete intersections.

Note that every weighted projective space $\PP$ is a Fano variety (with log terminal singularities).
Thus $\Aut(\PP)$ is a linear algebraic group. Alternatively, one can deduce this from the fact that
$\PP$ is a projective toric variety.
The following assertion is well known to experts.

\begin{proposition}\label{proposition:Aut-well-formed-WPS}
Suppose that the weighted projective space
$$
\PP=\PP(a_0^{r_0},\ldots,a_M^{r_M}),\quad a_0<\ldots<a_M,
$$
is well formed.
Let $R_U$ be the unipotent radical of the group $\Aut(\P)$, so that
$$
\Aut(\PP)\cong R_U\rtimes\Aut_{\mathrm{red}}(\PP),
$$
where the subgroup $\Aut_{\mathrm{red}}(\PP)$
is reductive.
Then $R_U$ consists of the automorphisms
\begin{multline}\label{eq:unipotent}
(x_{0,1}:\ldots:x_{0,r_0}:\ldots:x_{i,p}:\ldots:x_{M,1}:\ldots:x_{M,r_M})\mapsto\\
\mapsto
(x_{0,1}:\ldots:x_{0,r_0}:\ldots:x_{i,p}+\Phi_{i,p}:\ldots:x_{M,1}+\Phi_{M,1}:
\ldots:x_{M,r_M}+\Phi_{M,r_M}),
\end{multline}
where $x_{0,1},\ldots,x_{M,r_M}$ are weighted homogeneous coordinates on $\PP$, and
$\Phi_{i,p}$ is a weighted homogeneous polynomial of degree $a_i$ in the variables
$x_{0,1},\ldots,x_{i-1,r_{i-1}}$.
On the other hand, one has
\begin{equation}\label{eq:reductive}
\Aut_{\mathrm{red}}\big(\PP\big) \cong \big(\GL_{r_0}(\CC) \times \ldots \times \GL_{r_M}(\CC)\big)/\CC^*,
\end{equation}
where $\CC^*$ embeds into the above product by
\begin{equation}\label{eq:Gm-embedding}
t \mapsto (t^{a_0}{\mathrm{Id}}_{r_0},\ldots,t^{a_M}{\mathrm{Id}}_{r_M}),
\end{equation}
and $\mathrm{Id}_r$ denotes the identity $r\times r$-matrix.
Furthermore, for an appropriate choice of
weighted homogeneous coordinates $x_{0,1},\ldots,x_{M,r_M}$
the $(i+1)$-th factor in~\eqref{eq:reductive} acts by linear transformations of
the coordinates $x_{i,1},\ldots,x_{i,r_i}$.
\end{proposition}

\begin{proof}
The assertion about the reductive part of $\Aut(\PP)$
can be found in~\mbox{\cite[Proposition~A.2.5]{ProkhorovShramov-JCr3}}.
The assertion about the unipotent part
can also be deduced from the proof of \cite[Proposition~A.2.5]{ProkhorovShramov-JCr3}. Namely,
one can see that~\mbox{$\Aut(\PP)$}
is isomorphic to the quotient $\widetilde{\Aut}(\PP)/\CC^*$,
where $\CC^*$ is the torus whose action on
$$
\AA=\AA^{r_0+\ldots+r_M}
$$
is defined by the weights
$a_0,\ldots,a_M$, and $\widetilde{\Aut}(\PP)$ is the normalizer of~$\CC^*$ in the stabilizer of the point
$0\in\AA$ in $\Aut(\AA)$. In particular, $R_U$ is isomorphic to the unipotent radical of $\widetilde{\Aut}(\PP)$.
Furthermore, the group $\widetilde{\Aut}(\PP)$ is isomorphic
to the group of graded automorphisms of the Cox ring of $\PP$,
which can be identified with the coordinate algebra of~$\AA$ with the grading
defined by the weights $a_0,\ldots,a_M$.
The latter is just the polynomial algebra in~\mbox{$r_0+\ldots+r_M$} variables with the grading
defined by the weights. Now finding the unipotent radical
of~\mbox{$\widetilde{\Aut}(\PP)$} is straightforward.
\end{proof}

Note that Proposition~\ref{proposition:Aut-well-formed-WPS} fails without the well-formedness assumption, as shown by an example
of a weighted projective line
$$
\PP=\PP(1,2)\cong\PP^1.
$$

To work with reductive subgroups of automorphism
groups of weighted projective spaces, we will need
the following auxiliary assertion.

\begin{lemma}\label{lemma:choice-of-coordinates}
Let $\PP=\PP(a_0^{r_0},\ldots,a_M^{r_M})$,
where $a_0<\ldots<a_M$, be a well formed
weighted projective space.
Let $\Delta$ be a reductive subgroup
of $\Aut(\PP)$. Then
one can choose the weighted homogeneous coordinates
$x_{0,1},\ldots,x_{M,r_M}$ on $\PP$ so that
$\Delta$ is contained in a subgroup
$\Aut_{\mathrm{red}}(\PP)$ of $\Aut(\PP)$
described in Proposition~\ref{proposition:Aut-well-formed-WPS}.
\end{lemma}

\begin{proof}
Every reductive subgroup of a linear algebraic group is contained in a maximal
reductive subgroup, and all maximal reductive subgroups are conjugate to each other;
see for instance \cite[Theorems~6.4.4 and~6.4.5]{OnishikVinberg}.
Thus $\Delta$ is conjugate to a subgroup of~\mbox{$\Aut_{\mathrm{red}}(\PP)$}.
\end{proof}

We are going to show that certain subgroups of
$\Aut(\PP)$ act faithfully on weighted complete intersections in~$\PP$.

\begin{lemma}\label{lemma:Aut-not-cone-reductive}
Let $\PP$ be a well formed weighted projective space,
and let $X\subset\PP$ be an irreducible positive-dimensional
weighted complete intersection which is not an intersection with a linear
cone.
Let $\Delta$ be a reductive subgroup of $\Aut(\PP)$ that fixes every point of~$X$.
Then $\Delta$ is trivial.
\end{lemma}

\begin{proof}
Write $\PP=\PP(a_0^{r_0},\ldots,a_M^{r_M})$, where $a_0<\ldots<a_M$.
By Lemma~\ref{lemma:choice-of-coordinates}
one can choose weighted homogeneous coordinates
$x_{0,1},\ldots,x_{M,r_M}$ on $\PP$ so that
the action of $\Delta$ on the coordinate $x_{i,p}$ depends
only on coordinates $x_{i,1},\ldots,x_{i,r_i}$
of the same weight~$a_i$.
This provides $\Delta$-equivariant rational projections
$$
\psi_i\colon \PP\dashrightarrow \PP_i=
\operatorname{Proj}\big(\CC\left[x_{i,1},\ldots,x_{i,r_i}\right]\big)
\cong\PP^{r_i-1}.
$$

Since $X$ is not an intersection
with a linear cone, none
of the weighted homogeneous coordinates $x_{i,p}$ vanishes on $X$, see
Remark~\ref{remark:coordinate}.
Therefore, $X$ contains points where the projections
$\psi_i$ are regular. Let $Y_i\subset\PP_i$
be the (closure of) the image $\psi_i(X)$.
Since $X$ is irreducible, we conclude (again using
Remark~\ref{remark:coordinate}) that $Y_i$ is not contained in a
hyperplane in $\PP_i$. On the other hand, the action of $\Delta$ on
$Y_i$ is trivial by assumption. Hence the action of $\Delta$ on $\PP_i$
is trivial.

Therefore, each element $\delta\in\Delta$ acts
on $\PP$ by a transformation of the form
\begin{multline*}
\delta\colon (x_{0,1}:\ldots:x_{0,r_0}:\ldots: x_{i,p}:\ldots:
x_{M,1}:\ldots:x_{M,r_M})\mapsto\\
\mapsto
(\lambda_0 x_{0,1}:\ldots:\lambda_0 x_{0,r_0}:\ldots: \lambda_i x_{i,p}:\ldots:
\lambda_M x_{M,1}:\ldots:\lambda_M x_{M,r_M}),
\end{multline*}
where $\lambda_0,\ldots,\lambda_M$ are complex numbers.
As we saw above, one can choose a point $P$ on~$X$
such that none of the coordinates $x_{0,1},\ldots,x_{M,r_M}$ vanishes
at $P$. Since $\delta(P)=P$ by assumption, we conclude that
$$
\lambda_0^{a_0}=\ldots=\lambda_M^{a_M},
$$
which in turn means that the transformation $\delta$ is trivial,
cf.~\eqref{eq:Gm-embedding}.
\end{proof}

The assertion of Lemma~\ref{lemma:Aut-not-cone-reductive} fails
without the assumption that $X$ is not an intersection with a linear
cone even in the case when $\PP\cong\PP^N$ (where this assumption
is equivalent to the requirement that $X$ is not contained in a hyperplane).

\begin{lemma}\label{lemma:unipotent-Aut-not-cone-hypersurface}
Let $\PP$ be a well formed weighted projective space,
and let $X\subset\PP$ be an irreducible positive-dimensional
weighted complete intersection of multidegree $(d_1,\ldots,d_k)$.
Suppose that one has $a_i<d_j$ for all $i$ and $j$.
Let $\Delta$ be a subgroup of the unipotent radical of
$\Aut(\PP)$ that fixes every point of $X$.
Then $\Delta$ is trivial.
\end{lemma}

\begin{proof}
Let $f_1=\ldots=f_k=0$ be the equations of $X$ in $\PP$, so that $d_j=\deg f_j$.
We may assume that $a_0\le\ldots\le a_N$ and $d_1\le\ldots\le d_k$.
By assumption we have $d_1>a_N$.

We know from Proposition~\ref{proposition:Aut-well-formed-WPS}
that $\Delta$ consists of the elements of the form~\eqref{eq:unipotent}. If such an element preserves every point of $X$, then in the notation of~\eqref{eq:unipotent}
the polynomials~$\Phi_{i,p}$ must
be contained in the homogeneous ideal $I$
of $X$ in $\CC[x_0,\ldots,x_N]$. However,
the ideal~$I$ is generated by the elements $f_1,\ldots,f_k$ whose degrees are all greater than $a_N$, while the degrees of the polynomials
$\Phi_{i,p}$ do not exceed $a_N$. This means that all $\Phi_{i,p}$
must be zero polynomials, and thus the group $\Delta$ is trivial.
\end{proof}

The assertion of Lemma~\ref{lemma:unipotent-Aut-not-cone-hypersurface} fails
without the condition on the degrees.

\begin{example}
Consider a
weighted projective space $\PP=\PP(1^N,m)$
with weighted homogeneous coordinates $x_0,\ldots,x_N$,
where $N\ge 3$ and $m\ge 2$.
Let $X$ be a weighted complete intersection
in $\PP$ given by equations
$$
f_2=f_{2m}=0,
$$
where $f_2$ and $f_{2m}$ are
general weighted homogeneous polynomials in $x_i$'s
of degrees~$2$ and~$2m$, respectively.
Then $X$ is smooth and well formed, and it is a Fano variety
if~\mbox{$N\ge m+3$}.
Consider a homogeneous polynomial $g$ of degree $m-2$
in the variables~\mbox{$x_0,\ldots,x_{N-1}$}
and the (non-trivial) automorphism
$$
(x_0:\ldots:x_N) \mapsto  (x_0:\ldots:x_{N-1}:x_N+f_2g)
$$
of $\PP$.
Obviously, it acts trivially on $X$. Since there is a
$\binom{N+m-3}{N-1}$-dimensional space of polynomials of degree $m-2$
in the variables $x_0,\ldots,x_{N-1}$,
the subgroup of $\Aut(\PP)$ that fixes every point of $X$
contains a subgroup isomorphic to~\mbox{$(\CC^+)^{\binom{N+m-3}{N-1}}$}.
\end{example}

Lemma~\ref{lemma:unipotent-Aut-not-cone-hypersurface} implies the following
convenient corollary.

\begin{corollary}\label{corollary:unipotent-Aut-not-cone-hypersurface}
Let $X\subset\PP$ be a quasi-smooth well formed positive-dimensional
weighted hypersurface which is not an intersection with a linear cone.
Then
a subgroup of the unipotent radical of $\Aut(\PP)$
that fixes every point of $X$ is trivial.
\end{corollary}

\begin{proof}
Note that $X$ is irreducible by Lemma~\ref{lemma:WCI-connected}.
Let $f=0$ be the equation of $X$ in $\PP$, and set $d=\deg f$.
We may assume that~\mbox{$a_0\le\ldots\le a_N$}.

Suppose that $d<a_N$. Then
the polynomial $f$ does not depend on $x_N$.
Thus the point
$$
P=(0:\ldots:0:1)
$$
is contained in $X$.
Since $X$ is not an intersection with a linear cone, the polynomial $f$ does not contain
monomials of the form $\alpha x_i$, where $x_i$ is one of the weighted homogeneous coordinates on $\PP$, and $\alpha\in \CC$.
Hence all the partial derivatives of $f$ vanish at~$P$.
This means that $X$ is not quasi-smooth.
The obtained contradiction shows that $d\ge a_N$. Furthermore, since $X$ is not an intersection with a linear cone,
we see that $d>a_N$. Now the assertion follows from Lemma~\ref{lemma:unipotent-Aut-not-cone-hypersurface}.
\end{proof}

\section{Restriction of automorphisms}
\label{section:automorphisms-restricted}

In this section we prove Theorem~\ref{theorem:induced-action}.

\begin{proof}[Proof of Theorem~\ref{theorem:induced-action}]
We mostly follow the proof of \cite[Lemma~A.2.13]{ProkhorovShramov-JCr3}.

Denote by~$A$ the class of $\mathcal{O}_X(1)$ in the group $\Cl(X)$.
If $\dim X\ge 3$, then $A$ is an ample generator of $\Cl(X)$ by Theorem~\ref{theorem:Okada}, and
thus is invariant with respect to the group~\mbox{$\Aut(X)$}.
If the dimension of $X$ is arbitrary,
we note that the class of $\omega_X$ in $\Cl(X)$, i.e. the canonical class~$K_X$, is invariant
with respect to $\Aut(X)$.
Thus if $\dim X=2$ and~\mbox{$K_X\neq 0$}, or if~$X$ is a rational curve,
it follows from
Theorems~\ref{theorem:Okada} and~\ref{theorem:adjunction} that $A$
is invariant with respect to $\Aut(X)$ as well.
Therefore, in each case that we have to consider the ample class $A$ is invariant
with respect to $\Aut(X)$. This implies that~\mbox{$\Aut(X)$} is a linear algebraic group.

We have seen that the class $A$ of $\mathcal{O}_X(1)$ in $\Cl(X)$ is invariant with respect to
the group~\mbox{$\Aut(X)$}, and in particular with respect to $\Gamma$.
Set
\begin{equation*}
R(\PP)_m=H^0\big(\PP, \mathcal{O}_{\PP}(m)\big) \quad \text{and}\quad
R(X,A)_m=H^0\big(X, \mathcal{O}_X(mA)\big).
\end{equation*}
Then
\[
R(\PP)=\bigoplus\limits_{m=0}^{\infty} R(\PP)_m \quad \text{and}\quad
R(X,A)=\bigoplus\limits_{m=0}^{\infty} R(X,A)_m
\]
have natural structures of graded algebras.
Since $\mathcal{O}_{\PP}(1)$ and $A$ are ample, the algebras~\mbox{$R(\PP)$}
and~\mbox{$R(X,A)$} are finitely generated. One has
$$
\PP\cong\operatorname{Proj}\big(R(\PP)\big)
\quad \text{and}\quad
X\cong\operatorname{Proj}\big(R(X,A)\big).
$$
Recall that the restriction map
$$
\rho_m\colon R(\PP)_m\to R(X,A)_m
$$
is a surjection for every $m\ge 0$, see
\cite[Corollary~3.3]{PrzyalkowskiShramov-AutWCI}.

For every positive integer $K$ we define graded vector subspaces
\begin{equation*}
R(\PP)_{\le K}=\bigoplus\limits_{m\le K} R(\PP)_m\subset R(\PP)
\end{equation*}
and
\begin{equation*}
R(X,A)_{\le K}=\bigoplus\limits_{m\le K} R(X,A)_m\subset R(X,A).
\end{equation*}
Let
$$
U_m(\PP)\subset R(\PP)_m
$$
be the intersection
of $R(\PP)_m$ with the subalgebra of~\mbox{$R(\PP)$} generated
by~\mbox{$R(\PP)_{\le m-1}$}, and let
$$
U_m(X)\subset R(X,A)_m
$$
be the intersection
of $R(X,A)_m$ with the subalgebra of~\mbox{$R(X,A)$} generated
by~\mbox{$R(X,A)_{\le m-1}$}.
One has
$$
\rho_m(U_m(\PP))=U_m(X).
$$

There exists a  central extension
$\widetilde{\Gamma}$ of the group $\Gamma$ by a finite cyclic group, and an action of
$\widetilde{\Gamma}$  on $R(X,A)$ that induces the initial action of
$\Gamma$ on~$X$, see~\mbox{\cite[Lemma~A.2.11]{ProkhorovShramov-JCr3}}.
In particular, the group $\widetilde{\Gamma}$ acts on every
vector space $R(X,A)_m$.
Obviously, the subspace~$U_m(X)$ is $\widetilde{\Gamma}$-invariant.
Choose $V_m(X)\subset R(X,A)_m$ to be a $\widetilde{\Gamma}$-invariant
vector subspace such that
\begin{equation*}
U_m(X)\oplus V_m(X)= R(X,A)_m.
\end{equation*}
This is possible because the group $\Gamma$ is assumed to be reductive,
so that the group $\widetilde{\Gamma}$ is also reductive, and thus its
representation $R(X,A)_m$ is completely reducible.
Let $V_m(\PP)$ be a vector subspace
of $R(\PP)_m$ that is mapped to $V_m(X)$ isomorphically by $\rho_m$;
note that
$$
U_m(\PP)\cap V_m(\PP)=0
$$
by construction. Since $X$ is not an intersection with a linear cone,
it follows from Remark~\ref{remark:coordinate}
that
$$
R(\PP)_m=U_m(\PP)\oplus V_m(\PP).
$$
Note that for $m\gg 0$ the vector space $V_m(\PP)$ is zero.

Define the action of $\widetilde{\Gamma}$ on $V_m(\PP)$ so that
the isomorphism
$$
\rho_m\vert_{V_m(\PP)}\colon V_m(\PP)\stackrel{\sim}\longrightarrow V_m(X)
$$
is $\widetilde{\Gamma}$-equivariant.
Since the vector space $U_m(\PP)$ is the $m$-th graded component of the algebra generated by $R(\PP)_{\le m-1}$,
we obtain the action of $\widetilde{\Gamma}$ on $U_m(\PP)$ and $R(\PP)_m$
proceeding by induction on~$m$.
In other words, we identify $R(\PP)$ with the symmetric algebra of the graded vector space
$$
V(\PP)=\bigoplus\limits_{m=1}^{\infty} V_m(\PP),
$$
and the action of $\widetilde{\Gamma}$ on
$R(\PP)$ comes from its action on~$V(\PP)$.

Therefore, starting with the
$\widetilde{\Gamma}$-action on $R(X,A)$ that corresponds to the initial
$\Gamma$-action on~$X$, we have defined an action of $\widetilde{\Gamma}$
on $R(\PP)$ so that the restriction map
\begin{equation*}
\rho\colon R(\PP)\to R(X,A)
\end{equation*}
is $\widetilde{\Gamma}$-equivariant.
One can see from the construction that
the kernel of the projection~\mbox{$\widetilde{\Gamma}\to\Gamma$}
is contained in the
subgroup of $\Aut(R(\PP))$ acting as in~\eqref{eq:Gm-embedding};
this means that the action of~$\widetilde{\Gamma}$ on $\PP$
factors through $\Gamma$.
In other words, we have defined an action of $\Gamma$
on the weighted projective space
$\PP$ so that the embedding~\mbox{$X\hookrightarrow\PP$}
is $\Gamma$-equivariant.
\end{proof}

The assertion of Theorem~\ref{theorem:induced-action} fails
without the assumption that $X$ is not an intersection with a linear cone.

\begin{example}
Let $X$ be a line in $\PP=\PP^2$.
Both $X$ and $\PP^2$ have a faithful action of the alternating group
$\mathfrak{A}_5$. However, the action of $\mathfrak{A}_5$
on $X$ is not induced from its action on~$\PP^2$, because the latter comes from an irreducible
representation of~$\mathfrak{A}_5$.
Note that this representation has a quadratic invariant, which means that there
is a smooth $\mathfrak{A}_5$-invariant conic $X'$ in~$\PP^2$;
the actions of $\mathfrak{A}_5$ on $X'\cong\PP^1$ and $\PP^2$ agree with each other.
\end{example}

Also, the assertion of Theorem~\ref{theorem:induced-action} fails
for non-rational one-dimensional complete intersections
(cf. Remark~\ref{remark:curve-torsion}).

\begin{example}\label{example:elliptic-fail}
Let $C$ be a smooth cubic curve in $\PP^2$. Then the group $\Aut(C)$
contains finite subgroups of arbitrarily large order
while the stabilizer of $C$ in $\Aut(\PP^2)$ is finite.
\end{example}

Note that Theorem~\ref{theorem:induced-action}
still holds for smooth plane curves of genus greater than~$1$.
This follows from a theorem of Noether which states
that for such a curve the linear system defining an embedding
into $\PP^2$ is unique,
see for instance~\cite[Lemma~2.1]{Tyu75}
or~\mbox{\cite[Theorem~2.1]{Hartshorne-curves}}.
Similarly, Theorem~\ref{theorem:induced-action}
holds for smooth complete intersection curves
of genus greater than~$1$ in~$\PP^3$ (and also for
some smooth complete intersection curves in~$\PP^4$),
see~\mbox{\cite[Corollary~2.5 and Theorem~2.6]{CL}}.
We do not know if this is also the case for other
one-dimensional weighted complete intersections
of genus greater than~$1$. Similarly, we do not know
if Theorem~\ref{theorem:induced-action} can be generalized to the case
of weighted complete intersection surfaces with trivial canonical
class.

\section{Infinite automorphism groups}
\label{section:infinite-examples}

In this section we show by examples
that a quasi-smooth well formed Fano weighted complete intersection of arbitrary
dimension may have an infinite and even non-reductive
automorphism group.

\begin{example}
Let $a$ be any positive integer, and let $\PP=\PP(1^{N-1},a,a)$, where $N>2$. Consider the weighted hypersurface $X$ in $\PP$
given by the equation
$$
x_{N-1}x_N+F(x_0,\ldots,x_{N-2})=0,
$$
where $F$ is a general polynomial of degree $2a$ in $N-1$ variables.
Then $X$ is well formed and quasi-smooth,
and is not an intersection with a linear cone.
Furthermore,
$X$ is a Fano variety by Theorem~\ref{theorem:adjunction}.
The hypersurface $X$ is preserved by the subgroup $\CC^*\subset\Aut(\PP)$
such that the action of $t\in\CC^*$ is given by
$$
t\colon (x_0:\ldots:x_{N-2}:x_{N-1}:x_N)\mapsto (x_0:\ldots:x_{N-2}:tx_{N-1}:t^{-1}x_N).
$$
Now Lemma~\ref{lemma:Aut-not-cone-reductive}
implies (and one can also see this directly) that the
latter group acts faithfully on~$X$.
\end{example}

\begin{example}\label{example:non-reductive}
Let $a$ be any positive integer, and let $\PP=\PP(1^{N-1},a,a)$, where $N>2$. Consider the weighted hypersurface $X$ in $\PP$
given by the equation
$$
x_{N-3}x_{N-1}+x_{N-2}x_N+F(x_0,\ldots,x_{N-4})=0,
$$
where $F$ is a general polynomial of degree $a+1$ in $N-3$ variables.
Then $X$ is well formed and quasi-smooth, and is not an intersection with a linear cone.
Furthermore, $X$ is a Fano variety by Theorem~\ref{theorem:adjunction}.
Fix any polynomial $\Phi$ of degree $a-1$ in the variables~\mbox{$x_0,\ldots,x_{N-2}$}. Then there is an action of the group $\CC^+$ on $X$ such that $\alpha\in\CC^+$ acts by
$$
\alpha\colon (x_0:\ldots:x_{N-2}:x_{N-1}:x_N)\mapsto (x_0:\ldots:x_{N-2}:x_{N-1}+\alpha x_{N-2} \Phi:x_N-\alpha x_{N-3} \Phi).
$$
Thus $X$ is preserved by a subgroup $\Theta\subset\Aut(\PP)$ isomorphic to $\left(\CC^+\right)^s$, where
$$
s=\binom{a+N-3}{N-2}.
$$
One can see from Proposition~\ref{proposition:Aut-well-formed-WPS} that $\Theta$ is contained
in the unipotent radical~$R_U$ of~\mbox{$\Aut(\PP)$}.
Thus Corollary~\ref{corollary:unipotent-Aut-not-cone-hypersurface}
implies  that $\Theta$ acts faithfully on~$X$.
Now let $\Gamma\subset \Aut(\PP)$ be the (linear algebraic)
subgroup of $\Aut(\PP)$ that consists of \emph{all} automorphisms preserving~$X$.
Since $\Theta\subset \Gamma$, we see that the intersection of $\Gamma$ with~$R_U$ is non-trivial, and hence~$\Gamma$ is not reductive.
This implies that the linear algebraic
group~\mbox{$\Aut(X)$} is not reductive as well.
Indeed, otherwise by Theorem~\ref{theorem:induced-action}
the whole group $\Aut(X)$ is a quotient of $\Gamma$.
Thus the unipotent radical
of $\Gamma$ must act trivially on~$X$, which is not the case
because the action of its subgroup $\Theta$ is faithful.
\end{example}

As we have just seen, the image of the stabilizer of a quasi-smooth
well formed Fano weighted hypersurface $X\subset\PP$ under the restriction map
may be infinite. This is impossible for Calabi--Yau hypersurfaces.

\begin{proposition}\label{proposition:CY-restriction}
Let $X\subset\PP$ be a quasi-smooth
well formed weighted hypersurface with~\mbox{$i_X=0$}.
Let $\Gamma$ be the stabilizer of $X$ in $\Aut(\PP)$.
Then the image of $\Gamma$ in the group~\mbox{$\Aut(X)$} is
finite.
\end{proposition}

\begin{proof}
Note that $\Gamma$ is a linear algebraic group, and its image $\bar{\Gamma}$
in $\Aut(X)$ is a linear algebraic group as well. (Actually,
the whole group $\Aut(X)$ is also a linear algebraic group if the dimension
of $X$ is at least $3$, but we do not need this for the proof.)
On the other hand, the variety $X$ is not uniruled by
Proposition~\ref{proposition:rational-curves}(iii).
Hence the linear algebraic group $\bar{\Gamma}$ is finite, see for instance
\cite[Theorem~14.1]{Ueno}.
\end{proof}

We do not know whether Proposition~\ref{proposition:CY-restriction}
can be generalized to the case when $i_X<0$, cf.
Example~\ref{example:Prokhorov}.

\end{document}